\documentclass[12pt]{article}

\usepackage{amsmath, amsthm, amssymb}
\usepackage{color}

\begin{document}
\newtheorem{Theorem}{Theorem}[section]
\newtheorem{Proposition}[Theorem]{Proposition}
\newtheorem{Lemma}[Theorem]{Lemma}
\newtheorem{Example}[Theorem]{Example} 
\newtheorem{Corollary}[Theorem]{Corollary}
\newtheorem{Fact}[Theorem]{Fact}
\newtheorem{Conjecture}[Theorem]{Conjecture}
\newtheorem{Remark}[Theorem]{Remark}

\newenvironment{Definition} {\refstepcounter{Theorem} \medskip\noindent
 {\bf Definition \arabic{section}.\arabic{Theorem}}\ }{\hfill}

\newenvironment{Remarks} {\refstepcounter{Theorem}
\medskip\noindent {\bf Remarks
\arabic{section}.\arabic{Theorem}}\ }{\hfill}

\newenvironment{Exercise} {\medskip\refstepcounter{Theorem}
     \noindent {\bf Exercise
\arabic{section}.\arabic{Theorem}}\ } {\hfill}

\newenvironment{Proof}{{\noindent \bf Proof\ }}{\hfill}

\newenvironment{claim} {{\smallskip\noindent \bf Claim\ }}{\hfill}

\def \L {{\cal L}}
\def \diamond {$\diamondsuit$}
\def \B {{\cal B}}
\def \mod {{\rm mod \ }}
\def \iso {\cong}
\def \Lor  {\L_{\rm or}}
\def \Lr {\L_{\rm r}}
\def \Lg {\L_{\rm g}}
\def \I {{\cal I}}
\def \M {{\cal M}}\def\N {{\cal N}}
\def \E {{\cal E}}
\def \Proj {{\mathbb P}}
\def \H {{\cal H}}
\def \x {\times}
\def \Stab {{\rm Stab}}
\def \Z {{\mathbb Z}}
\def \V {{\mathbb V}}
\def \C {{\mathbb C}} \def \Cexp {\C_{\rm exp}}
\def \R {{\mathbb R}}
\def \Q {{\mathbb Q}}\def \K {{\mathbb K}}
\def \F {{\cal F}}\def \A {{\mathbb A}}
\def \G {{\cal G}} \def \X {{\mathbb X}}
\def \GG {{\mathbb G}}
\def\HH {{\mathbb H}}
\def \Nn {{\mathbb N}}\def \Nn {{\mathbb N}}
\def\D {{\mathbb D}}
\def \hat {\widehat}
\def \bar{\overline}
\def \Spec {{\rm Spec}}
\def \bul {$\bullet$\ }
\def\proves {\vdash}
\def \Co {{\cal C}}
\def \ACFp {{\rm ACF}_p}
\def \ACF0 {{\rm ACF}_0}
\def \ee {\prec}
\def \Diag {{\rm Diag}}
\def \Diage {{\rm Diag}_{\rm el}}
\def \DLO {{\rm DLO}}
\def \d {{\rm depth}}
 \def \dist {{\rm dist}}
\def \P {{\cal P}}
\def \ds {\displaystyle}
\def \Fp {{\mathbb F}_p}
\def \acl {{\rm acl}}
\def \dcl {{\rm dcl}}

\def \dom {{\rm dom}}
\def \tp {{\rm tp}}
\def \stp {{\rm stp}}
\def \Th  {{\rm Th}}
\def\< {\Lngle}
\def \> {\rangle}
\def \n {\noindent}
\def \minusdot{\hbox{\ {$-$} \kern -.86em\raise .2em \hbox{$\cdot \
$}}}
\def\exp {{\rm exp}}\def\ex {{\rm ex}}
\def \td {{\rm td}\ }
\def \ld {{\rm ld}}
\def \span {{\rm span}}
\def \tilde {\widetilde}
\def \d {\partial}
\def \del {\partial}
\def \cl {{\rm cl}}
\def \acl {{\rm acl}}
\def \cN {{\cal N}}
\def \Qalg {{\Q^{\rm alg}}}
\def \th {^{\rm th}}
\def \deg { {\rm deg} }
\def\hat {\widehat}
\def\li {\L_{\infty,\omega}}
\def\lo {\L_{\omega_1,\omega}}
\def\lk {\L_{\kappa,\omega}}
\def \ee {\prec}\def \bSigma {{\mathbf\Sigma}}
\def \mod {\ {\rm mod\ }}
\def \K{{\mathbb K}}
\def \X {{\bf X}}
\def \c {{\bf c}}
\def \d {{\bf d}}
\def \I {{\cal I}}
\def \alg {{\rm alg}}
\def \SS {{\rm SS}}
\def \blue {\color{blue}}
\def\red{\color{red}}
\def \supp {{\rm supp}}
\def \< {\langle}
\def \> {\rangle}
 \def \lc{\lceil}
 \def \rc{\rceil}
 
\title{Uncountable Real Closed Fields with {\rm PA} Integer Parts}
\author{David Marker\thanks{Partially supported by NSF grant  DMS-0653484.}\\ University of Illinois at Chicago\and James H. Schmerl\\University of Connecticut \and Charles Steinhorn\thanks{Partially supported by NSF grant DMS-0801256.}\\ Vassar College}
\maketitle

\begin{abstract} D'Aquino, Knight and Starchenko classified the countable
real closed fields with integer parts that are nonstandard models of Peano
Arithmetic.  We rule out some possibilities for extending their results to the 
uncountable and study real closures of $\omega_1$-like models of PA.\end{abstract}

If $K$ is a real closed field we say that a subring $M$ of $K$ is an {\em integer part} of $K$ if $M$ is discretely ordered, i.e., there is no element $m\in M$ with $0<m<1$, and for every $x\in K$ there is $m\in M$ such that $x\le m < x+1$.  
A surprising theorem of Mourgues and Ressayre \cite{MR} tells us that every real closed field has an integer part.

Integer parts satisfy a very weak fragment of Peano Arithmetic (PA).  By {\em Open Induction} we mean the fragment of PA in the language $\{+,-,\cdot, <\nobreak, 0, 1\}$ in which we restrict the induction schema to quantifier-free formulas. Shepherdson \cite{S} showed that the set of nonnegative elements of an ordered ring is a model of open induction if and only if the ring is an integer part of its real closure.   Open Induction is a very weak fragment; indeed, Shepherdson showed that it is too weak to prove the irrationality of  $\sqrt {2}$.  D'Aquino, Knight and Starchenko \cite{DKS} investigated which real closed fields have an integer part whose nonnegative elements form a nonstandard model of PA.

Recall that a structure $\M$ in a finite language $L$ is {\em resplendent} if and only if for every finite $L^*\supset L$, recursive $L^*$-theory $T(\bar w)$ with free variables $\bar w=(w_1,\ldots ,w_k)$, and tuple $\bar a\in M^k$, if 
$T(\bar a) \cup \Th (\M, \bar a)$ is consistent, then there is an expansion of $\M$ to an $L^*$-structure 
$(\M^*, \bar a)\models T(\bar w)$. A structure $\M$ is {\em recursively saturated} if for every recursive set of formulas $\gamma(v,\bar w)$ in the free variables $v$ and $\bar w=(w_1,\ldots ,w_k)$, if $\bar a\in M^k$ and $\gamma(v,\bar a)$ is consistent with the elementary diagram of $\M$ then $\gamma$ is realized in $\M$.  Barwise and Schlipf \cite{BS} showed that, in a finite language, every resplendent model is recursively saturated and every countable recursively saturated model is resplendent.

It is easy to see that every resplendent real closed field has an integer part that is a model of PA.  D'Aquino, Knight and Starchenko showed that every real closed field with a nonstandard model of PA as an integer part is recursively saturated.
Thus a countable real closed field has an integer part that is a nonstandard model of PA if and only if it is recursively saturated. Ko\l odziejczy and Je\v{r}\'abek \cite{KJ} generalized  this to show that a real closed field must be recursively saturated even to have an integer part that is a nonstandard model of IE$_2$, where  IE$_2$
is the fragment of I$\Delta_0$ in which induction is allowed just for formulas beginning with a string of bounded existential quantifiers followed by a string of bounded universal quantifiers.

Is there a natural characterization of the uncountable real closed fields with nonstandard models of PA for integer parts?
We show that two natural possibilities do not work.  In \S \ref{recsat} we show that there are recursively saturated (indeed even $\aleph_1$-saturated) real closed fields with no model of PA as an integer part and in \S \ref{resp} we show that the real closure of an $\omega_1$-like model of PA is not resplendent.   

One interesting consequence of \cite{DKS} is that two countable nonstandard models of PA have isomorphic real closures if and only if they have the same standard systems.
Can this be generalized to $\omega_1$-like models?  In \S 3 we show that two $\omega_1$-like models of PA with 
the same standard system have real closures with isomorphic value groups, but in \S 4 we show that there are $2^{\aleph_1}$ $\omega_1$-like recursively saturated models of PA with the same standard 
system and pairwise non-isomorphic  real closures. In an earlier version of this paper the first and third authors proved this result from $\diamondsuit$; the second author showed how to eliminate this assumption. 

We now fix some notation and terminology that will be employed throughout the paper. 
If $\M$ is a domain, we let $Q(\M)$ denote the fraction field of $\M$ and $R(\M)$ denote its real closure. 

If $K$ is a real closed field and $O$ is a convex subring, there is a canonical valuation 
$v_O:K^\times\rightarrow \Gamma_O$ defined by $v_O(x)\le v_O(y)$ if and only if $y/x\in O$.
We let $k_O$ denote the residue field under $v_O$.
If $O$ is the convex subring of finite elements of $K$, we have the {\em standard valuation\/} $v:K^\times\rightarrow\Gamma$, and we let $v(K)$ denote the value group $\Gamma$ with respect to this valuation.  Note that if $\Gamma_0\subset \Gamma$ is the convex subgroup $\Gamma_0=\{v(x):v_O(x)=0\}$,
then $\Gamma_O\iso \Gamma/\Gamma_0$.

\smallskip
The authors are very grateful to Roman Kossak for several discussions on this material. The first author would also like to thank the CUNY Graduate Center for its hospitality during the 2011--12 academic year.

 \section {Saturation v. Integer Parts} \label{recsat}
 
 We begin by showing that there are $\aleph_1$-saturated real closed fields in which no integer part is a model of PA.
 
 If $(G,+,<)$ is an ordered abelian group, $k$ is a  field, and $t$ is an indeterminate 
 we can form the Hahn series field, $k((G))$ of formal sums
 $$f=\sum_{g\in G} a_g t^g,$$ where each $a_g\in k$ and the support of $f$, 
 $$\supp(f)=\{g:a_g\ne 0\}$$ 
 is well ordered by $<$.  We identify $t^0$ with $1$.
 Addition of two series is performed componentwise and multiplication is defined by 
 $$\left(\sum_{g\in G}a_gt^g\right)\left(\sum_{g\in G}b_gt^g\right)=\sum_{g\in G}\left(\sum_{g_1+g_2=g}a_{g_1}b_{g_2}\right) t^g.$$
This operation is well-defined and makes $k((G))$ into a field that carries a natural valuation given by $v(f)=\min\supp(f)$. Note that this valuation $v$ is indeed the standard valuation as defined above in the introduction. The following lemma summarizes the basic facts we need about $k((G))$.  See, for example \cite{N} for details.

 \begin{Lemma}\label{L1.1} {\rm (i)}\enspace If $k$ is real closed and $G$ is divisible, then $k((G))$ is a real closed field.
In this case, the unique ordering agrees with the lexicographic ordering and the infinitesimal elements are the
elements $\epsilon$ with $v(\epsilon)>0$.

{\rm (ii)}\enspace If $\sum_{n=0}^\infty a_nX^n$ is a formal power series over $k$ and $\epsilon\in k((G))$ satisfies $v(\epsilon)>0$, then $\sum_{n=0}^\infty a_n\epsilon^n$ is a well-defined element of $k((G))$.
\end{Lemma}
 
We show that we can choose $G$ so that that ${\mathbb R}((G))$ is $\aleph_1$-saturated but does not have an integer part that is a model of PA.  Thus recursive saturation is insufficient to guarantee an integer part that is a model of PA.  We need  the following folklore lemma. Sharper versions appear in \cite{KKMZ}, but we include the proof below for completeness.

\begin{Lemma} Let $G$ be an $\aleph_1$-saturated divisible ordered abelian group. Then the real closed Hahn series field 
$K=\R((G))$ is $\aleph_1$-saturated.
\end{Lemma}

\begin{proof}  It suffices by quantifier elimination for real closed fields to show that if 
 $$a_0<  \dots< a_n<\dots< b_n< \dots< b_0$$ with $a_n<b_m$ for all $n$ and $m$,
then there is an element $x\in \bigcap_{n=0}^\infty [a_n, b_n]$.  Let $v:K^\times\rightarrow G$ be the usual valuation and let
$\gamma_n=v(b_n-a_n)$.  Thinning the sequence if necessary there are three cases to consider.

\medskip {\n \bf Case 1.} $\gamma_0<\gamma_1<\cdots$ and $v(b_0-b_j)>\gamma_i$ for all $i,j$ (or, similarly, the case in which all the $a_i$ are very close to $a_0$). 

\smallskip
In this case we use the $\aleph_1$-saturation of $G$ to find $\gamma$ such that for all $i,j$
$\gamma_i <\gamma<v(b_0-b_j)$.   Let $x=b_0-t^\gamma$.  

Then
$$v(b_j-x)=\min (v(b_0-x), v(b_j-b_0))=\gamma \hbox { and } x<b_j$$ and
$$v(a_j-x)= \min (v(a_j-b_j), v(b_j-x))=\gamma_i \hbox{ and } a_j<x.$$  
Thus $x$ realizes the type $x\in \bigcap_{n=0}^\infty [a_n, b_n]$. 

\medskip {\n \bf Case 2.} $\gamma_0<\gamma_1<\cdots$ and   $v(b_{n+1}-b_n)=v(a_{n+1}-a_n) =\gamma_n$ for all $n$.

Suppose that $a_n=\sum r_{n,\gamma} t^\gamma$. Put $s_n=\sum_{\gamma<\gamma_n}r_{n,\gamma} t^\gamma$.
Then $s_n$ is an initial formal summand of $s_{n+1}$ and $s_{n+1}\in [a_n,b_n]$.  Let $s$ be the natural limit of the sequence of $s_n$ for $n\in\N$.  
Then $s\in \bigcap_{n=0}^\infty [a_n, b_n]$, as required.

\medskip \n {\bf Case 3.} There is an element $\gamma\in G$ such that $\gamma_n=\gamma$ for all $n$.  

Translating by $\sum_{\gamma'<\gamma}r_{\gamma'} t^{\gamma'}$ and multiplying by $t^{-\gamma}$ there are real numbers $c_n\ne d_n$ such that $a'_n=c_n(1+\epsilon_n)$ and $b'_n=d_n(1+\delta_n)$,
where $v(\epsilon_n), v(\delta_n)>0$.  Clearly $c_0\le c_1\le\cdots\le d_1\le d_0$.  Choose $r\in \R$ such that
$c_n\le r\le d_n$ for all $n$.
There are two subcases to consider.

\smallskip \n {\bf Subcase 3a.}    $c_n<r<d_n$ for all $n$.

Then $(\sum_{\gamma'<\gamma}r_{\gamma'} t^{\gamma'})+rt^\gamma\in \bigcap_{n=0}^\infty [a_n, b_n]$ as desired.

\smallskip \n {\bf Subcase 3b.} There is some $N$ such that $c_n<r=d_n$ for all $n\ge N$. (The case in which the $c_n$ are eventually constant is similar.)

In this case apply the $\aleph_1$-saturation of $G$ to find an element $\delta$ such that 
$$\gamma<\delta< v\left(b_n-\left(\sum_{\gamma'<\gamma}r_{\gamma'} t^{\gamma'}\right)+rt^\gamma\right)$$
for all $n\ge N$.  Then 
$$(\sum_{\gamma'<\gamma}r_{\gamma'} t^{\gamma'})+rt^\gamma-t^{\delta}$$ realizes the type, as required.
\end{proof}

\begin{Lemma} Suppose that $M\subset \R((G))$ is an integer part with $M\models {\rm PA}$. 
Then there is an exponential map $E:\R((G))\rightarrow \R((G))$, that is, a surjective homomorphism 
from the additive group of $\R((G))$ onto its multiplicative group of positive elements.
\end{Lemma} 

\begin{proof}   Let   $\mu$ denote the maximal ideal of infinitesimals of $\R((G))$ and  
$E_0:\mu\rightarrow 1+\mu$ be given by 
$$\epsilon\mapsto \sum_{n=0}^\infty
 \frac{\epsilon^n}{n!}.$$  The following properties of $E_0$ are well-known, see, for example \cite{DMM}:

\begin{itemize}
\item[(i)] $E_0$ is well-defined (by Lemma~\ref{L1.1}); 
\item[(ii)]  $E_0(x+y)=E_0(x)E_0(y)$;
\item[(iii)]  $E_0$ is surjective with inverse
 $$l(1+\epsilon)=\sum_{n=1}^\infty (-1)^{n+1}\epsilon^n.$$
\end{itemize}
\medskip
Since  $M\models$ PA   there is a definable function $m\mapsto 2^m$ on the positive 
elements of $M$ that extends 
exponentiation on the natural numbers and satisfies $2^{m+n}=2^m2^n$.  For $m\in M$ with $m<0$
put $2^{m}=1/2^{-m}$.  

We define an exponential function on $\R((G))$ as follows.  For every nonnegative $x\in \R((G))$ 
there is an element $m\in M$ such that $m\le x<m+1$.  Let $x=m+r+\epsilon$ where $r\in\R$,
$0\le r<1$ and $\epsilon\in \mu$.   Define 
$$E(x)= 2^m2^rE_0(\epsilon \ln 2).$$ 
Suppose that $x,y\in \R((G))$ where 
$x=m+r+\epsilon$ and $y=n+s+\delta$.  Then
$$x+y=\begin{cases} (m+n)+(r+s)+(\epsilon+\delta) & \mbox{if $(r+s+\epsilon+\delta<1)$} \\ 
(m+n+1)+(r+s-1)+ (\epsilon+\delta) & \mbox{otherwise}
\end{cases}$$
and
$$E(x+y)= \begin{cases}  2^{m+n}2^{r+s}+E_0(\ln 2(\epsilon+\delta)) & \mbox{if $(r+s+\epsilon+\delta<1)$} \\ 
\noalign{\smallskip}
2^{m+n+1}2^{r+s-1}+ E_0(\ln 2(\epsilon+\delta)) & \mbox{otherwise.}
\end{cases}$$
In either case, $E(x+y)=E(x)E(y)$. For $x<0$, define $E(x)=1/E(-x)$. 

Now let $y\in \R((G))$ such that $y\ge 1$. We find an $x\in \R((G))$ with $E(x)=y$. As $y\ge 1$ there is $m\in M$ such that $2^m\le y<2^{m+1}$ then we can find $r\in \R$ with  $1\le r< 2$ and $\epsilon\in \mu$ such that
$y=2^mr(1+\epsilon)$.  Set 
$$x=m+\ln r +\frac{l(1+\epsilon)}{\ln 2}.$$
Then $0\le \ln r< 1$ and $E(x)=y$.  If $0<y<1$, there is $x\in \R((G))$ such that $E(x)=1/y$ and thus $E(-x)=y$.  Thus $E$ is a surjective homomorphism of the additive group onto the multiplicative group of positive elements.
\end{proof}

\begin{Corollary} There is an $\aleph_1$-saturated real closed field such that no integer part is a model of {\rm PA}.
\end{Corollary}

\begin{proof}  Kuhlmann, Kuhlmann, and Shelah \cite {KKS} show that no Hahn field $\R((G))$ can support such an exponential.
\end{proof}

\medskip
Using further results from  \cite{KKMZ}, it is easy to extend the corollary to show for all uncountable $\kappa$ that $\kappa$-saturation is insufficient to guarantee existence of an integer part that is a model of PA. 

Refinements of the results in this section can be found in \cite{CDK}.

\section{Integer Parts v. Resplendence} \label{resp}

We next show that the real closure of an uncountable model of PA need not be resplendent. Recall that a linear order is 
$\omega_1$-like, if it is uncountable but every proper initial segment is countable.  

\begin{Proposition} \label{notresplendent}
If $\M$ is an $\omega_1$-like model of PA then $R(\M)$ is not resplendent.
\end{Proposition} 

\begin{proof}
Observe first that if $K$ is a real closed field and $\M$ and $\N$ are integer parts then $(M,<)$ is order isomorphic to $(N,<)$ via the map that sends an element of $M$ to its integer part relative to $\N$. It is easy to see that if $K$ is resplendent then it has an integer part with initial segments having cardinality $|K|$. 
\end{proof}

The next statement provides information about the value group (under the standard valuation) of a resplendent real closed field. This contrasts with Corollary~\ref{omegaonevaluegp} below. 
  
\begin{Proposition}\label{biginitsegment} Let $K$ be a resplendent real closed field. Then there is a convex subring $O$
such that $|k_O|=|K|$ and $(k_O,+,<)\iso \Gamma_O$.  In particular, if $\Gamma$ is the
value group of $K$ under the standard valuation, then $\Gamma$ has bounded intervals of cardinality $|K|$.
\end{Proposition}
\begin{proof}
We can write a sentence in the language of ordered fields with an extra predicate for the valuation ring
and a function from the field into the valuation ring such that fields satistfying this additional sentence have
the above property.

Let $F$ be a real closed subfield of $\R((t^\R))$ of cardinality $2^{\aleph_0}$
such that under the standard valuation the residue field is $\R$ and the value group is $(\R,+)$.
Then $F$ has the properties described in the Proposition.  Since this is true in some real closed field, it holds in every resplendent real closed field.

The value group $\Gamma_O$ has bounded intervals of cardinality $|K|$.  Since $\Gamma_O$ is 
a quotient of $\Gamma$ by a convex subgroup, the same is true for $\Gamma$.
\end{proof}

The following lemma is applied tacitly in the next section. 

\begin{Lemma}\label{convexvaluegp}  If $\M\ee_e\N$ are models of PA, then $v(R(\M))$ is a convex subgroup of $v(R(\N))$.
\end{Lemma} 

\begin{proof} Let $g\in v(R(\N))\setminus v(R(\M))$.  Without loss of generality, $g<0$, i.e.,
there is an infinite element $x\in R(\N)$ with $v(x)=g$. Since $\N$ is an integer-part of $R(\N)$,  there
is some $n\in \N$ such that $|x-n|<1$, and, as $x$ is infinite, $v(x)=v(n)$.  From $\M\ee_e\N$ it follows that $n>R(\M)$,
and hence $g<v(R(\M))$.
\end{proof}

Lemma~\ref{convexvaluegp} yields the following corollary. Combined with Proposition~\ref{biginitsegment} it gives a second proof of Proposition~\ref{notresplendent}. 

\begin{Corollary}\label{omegaonevaluegp}
Let $\M$ be an $\omega_1$-like model of {\rm PA}. Then $v(R(\M))$
is $\omega_1$-like, i.e., for any $g>0$, the set $\{h\in v(R(\M)): |h|<g\}$ is countable. 
\end{Corollary}

\section{Value groups of $\omega_1$-like models}\label{VG-sec} 

Let $G$ be a divisible ordered abelian group.  For $a,b\in G$, put $r(a,b)=\{q\in \Q:
qb<a\}$.  The {\em standard system} of $G$ is $\SS(G)=\{r(a,b): a,b\in G\}$.    
In a real closed field $K$, we   define the standard system to be $\SS(K)=\{r(a):a\in K\}$, where $r(a)=\{q\in \Q: q<a\}$.  
In a nonstandard model of PA the standard system is the set of $r(a)=\{n: $ the $n^{\rm th}$-prime divides $a\}$.
 
The goal of this section is to prove  

\begin{Theorem}\label{VG}  If $\M$ and $\N$ are $\omega_1$-like models of {\rm PA} with the same standard system, then their value groups are isomorphic.
\end{Theorem}

 We give a self contained  proof of Theorem \ref{VG} and also show that it follows from a result of Harnik on the structure of additive reducts of models of PA. 
Before turning to the proof, we collect several facts. 

\begin{Lemma}\label{StSystSat}  Let $G$ be a recursively saturated divisible ordered abelian group. Then:
\begin{itemize}
\item[{\rm (i)}] if $a,b,c\in G$, there is $d\in G$ such that $r(a,b)=r(d,c)$.

\item[{\rm (ii)}] $G$ is $\SS(G)$-saturated, i.e., any complete $n$-type realized in $G$ is in the Turing ideal generated by $\SS(G)$ and for any partial type $p(v,\bar w)$ recursive in an element of $\SS(G)$, if $\bar a\in G$ and $p(v,\bar a)$ is consistent with the elementary diagram of $G$, then $p$ is realized in $G$.\qed
\end{itemize}
\end{Lemma}

\begin{Lemma} If $K$ is a recursively saturated real closed field, then $v(K)$ is recursively saturated; indeed, $v(K)$ is $\SS(K)$-saturated.
\end{Lemma}
\begin{proof}  We can transform a type $p$  over $v(K)$ to a  type $q$ over $K$, since, for all $n$ and elements $x,y_1,\dots,y_n>0$ of $K$, we have $v(x)<\sum m_i v(y_i) $
if and only if $x>\prod y_i^{m_i}$, where $m_1,\ldots ,m_n\in\Q$.
 \end{proof}

Appealing to Theorem~2.1 and Proposition~3.1 in \cite{DKS} we have 

\begin{Corollary}\label{recsatVG}  If $\M\models {\rm PA}$ and $K$ is the real closure of $\M$, then $v(K)$ is 
recursively saturated.  Indeed, $v(K)$ is $\SS(\M)$-saturated.
\end{Corollary} 

We now fix some notation for the proof of Theorem~\ref{VG}. 

Let $\M$ and $\N$ be $\omega_1$-like models of {\rm PA} (with the same standard system) as in the statement of the theorem.   We can find continuous chains of  countable models 
$$\M_0\ee_e \M_1\ee_e\dots\ee_e \M_\alpha\ee_e\dots$$ 
and 
 $$\N_0\ee_e \N_1\ee_e\dots\ee_e \N_\alpha\ee_e\dots$$  
 with $\bigcup \M_\alpha=\M$ and $\bigcup\N_\alpha=\N$.
As elements of the standard system of a model of PA are coded arbitrarily low in the nonstandard part, we have 
$\SS(\M)=\SS(\M_0)$ is countable. Let $G_\alpha$ be the value group of $R(\M_\alpha)$ and 
$G^\prime_\alpha$ be the value group of $R(\N_\alpha)$.  

To prove the theorem we shall construct a continuous increasing sequence of isomorphisms
$\sigma_\alpha: G_\alpha\rightarrow G_\alpha^\prime$, where 
$\sigma_0\subset\sigma_1\subset\cdots\subset\sigma_\alpha\subset\cdots$. We require preliminary lemmas.

 \begin{Lemma}\label{coinitial} Let $G\subseteq H$ be divisible ordered abelian groups such that $G$ is convex in $H$, and 
 $H/G$ is nontrivial and finite dimensional over $\Q$. Then there exists an element $h\in H$ so that if $g\in H$ and
 $g>G$, then $mg>{h}$ for some   $m\in\Nn$.  Also, there is some $h\in H$ such that every element of $H$ is bounded above by $mh$ for some $m\in\Nn$.
 \end{Lemma} 
 
 \begin{proof} Suppose not. Then we can find $h_0>h_1>\cdots$ in $H$ such that each $h_i>G$ but $h_{n}>mh_{n+1}$ for all $m\in \Nn$. But this contradicts the fact that $H/G$ is 
 finite dimensional.  
 
 The proof of the second assertion is similar.
 \end{proof} 
 
\begin{Lemma} \label{gap} Suppose $g_1,\dots,g_n\in G_{\alpha+1}\setminus G_\alpha$.
Let  $H$ be the divisible hull of $G_\alpha\cup\{g_1,\dots,g_n\}$.
Then there exists $g\in G_{\alpha+1}$ such that $g>G_{\alpha}$ but $g<h$ for all $h\in H$ such that $h>G_\alpha$.
\end{Lemma} 
 \begin{proof}   By Lemma \ref{coinitial} there is $h\in H$ such that $h>G_{\alpha}$ and if $x\in H$ and $x>G_{\alpha}$, then $mx>h$ for some $m\in\Nn$.  We can find $a\in \M_{\alpha+1}\setminus \M_\alpha$ such that $v(1/a)=h$.  There is an element $b\in\M_{\alpha+1}$ such that $2^b\le a\le 2^{b+1}$.  Then $b>\M_\alpha$  and $ b<\sqrt[m] {a}$ for all nonzero $m\in\Nn$.  Thus $v(1/b)>G_{\alpha}$, but
 $v(1/b)<h/m$ for all nonzero $m\in \Nn$.  Let $g=v(1/b)$.
 \end{proof}
 
 \medskip The next lemma provides the main step in the proof of Theorem~\ref{VG}.
 
 \begin{Lemma}\label{VGmain}  Suppose 
 $\sigma_\alpha:G_\alpha\rightarrow G^\prime_\alpha$ is an isomorphism.  Then $\sigma_\alpha$ can be extended
 to an isomorphism $\sigma_{\alpha+1}:G_{\alpha+1}\rightarrow G_{\alpha+1}^\prime$.
  \end{Lemma}
 
 \begin{proof}  The isomorphism $\sigma_{\alpha +1}$ is built via a back and forth construction between the countable groups 
 $G_{\alpha+1}$ and $G_{\alpha+1}^\prime$. For the initial step choose $a\in G_{\alpha+1}$ with $a>G_{\alpha}$ and 
 $b\in G^\prime_{\alpha+1}$ with $b>G^\prime_\alpha$.  Let $H=G_{\alpha}\oplus \Q a$ and define $\sigma:H\rightarrow G_{\alpha+1}^\prime$ by 
 $$\sigma (g+ma)=\sigma_\alpha(g)+mb.$$  
 Then $\sigma$ is an order preserving embedding extending $\sigma$. 
 
In general,  suppose we have $G_{\alpha}\subset H\subset G_{\alpha+1}$, where $H$ is a nontrivial finite dimensional extension of $G_\alpha$ and $\sigma:H\rightarrow G^\prime_{\alpha+1}$ extending $\sigma_\alpha$. It suffices to show that if 
$a\in G_{\alpha+1}\setminus H$ we can extend $\sigma$ to $H\oplus \Q a$. Without loss of generality $a>0$. There are several cases to consider.

\medskip 
\n {\bf Case 1.}  $a>H$.  

By Lemma \ref {coinitial} there exists $h\in H$ such the elements $mh$, for $m\in\N$ are cofinal in $H$.
Since $G_{\alpha+1}^\prime$ is recursively saturated by Corollary~\ref{recsatVG}, there is $b\in G_{\alpha+1}$ such that
$b>m\sigma(h)$ for all $m\in \Nn$.  We then can extend $\sigma$ by setting $\sigma(a)=b$.

\medskip
\n {\bf Case 2.} Suppose there is $g\in H$ such that $a>g+G_{\alpha}$, but $a<g+h$ for all $h\in H$ such that $h>G_0$. 
(The case in which $a<g+G_{\alpha}$ is similar).

Translating by $-g$, we may assume $g=0$.  Lemma \ref{gap} provides an element $b\in G_{\alpha+1}^\prime$
such that $b>G_\alpha^\prime$ and $b<\sigma(h)$ for all $h\in H$ with $h>G_\alpha$.  We then extend $\sigma$ by
setting $\sigma(a)=b$. 

To determine the remaining case requires some preliminary analysis. Let $h_1,\dots,h_n$ be a basis for $H$ over $G_{\alpha}$.  Put 
$$C=\left\{\sum m_ih_i: m_i\in\Q\mbox{\ and\ }\sum m_i h_i < a\right\}$$ 
and 
$$D=\left\{\sum m_ih_i: m_i\in\Q{\ and\ }\sum m_ih_i>a\right\}.$$ 
Since $G_\alpha$ is convex in $G_{\alpha +1}$, if $c\in C$ then $c+g<a$ for all $g\in G_\alpha$, and likewise if $d\in D$ then 
$d+g>a$ for all $a\in G_\alpha$.  If $D$ is empty, then we are in Case~1. If $C$ is empty or has a greatest element, or $D$ has a least element, then we are in Case~2.  Thus we are left with 
 
 \medskip
 \n {\bf Case 3.} $C$ does not have a greatest element and $D$ does not have a least element. 
 
 Then $\tp(a/H)$ is determined by $\tp (a/C\cup D)$. Let $C^*=\{\bar m\in \Q^n:\sum m_ih_i\in C\}$
 and $D^*=\{\bar m\in \Q^n:\sum m_ih_i\in D\}$. We now apply Lemma~\ref{StSystSat} and Corollary~\ref{recsatVG}. 
 Since $G_{\alpha +1}$ is $\SS(G_{\alpha +1})$-saturated, 
 $C^*,D^*$ are recursive in elements of $\SS(G_{\alpha +1})$.  Thus, as $G^\prime_{\alpha+1}$ is 
 $\SS(G_{\alpha +1})$-saturated,
 there is an element $b\in G'_{\alpha+1}$ such that $\sum m_i \sigma(h_i)<b \Leftrightarrow \bar m\in C^*$.
 Thus we can extend $\sigma$ by setting $\sigma(a)=b$.

 The the proof of the lemma is now complete. 
 \end{proof}
 
 \noindent{\em Proof of Theorem~\ref{VG}.\/}\enspace 
The construction of the required isomorphism of value groups is now an immediate consequence of Lemma \ref{VGmain}. 
\qed 

\begin{Remark}
 {\rm There are only $2^{\aleph_0}$ possible countable standard systems, so only $2^{\aleph_0}$
 possible value groups for the real closure of an $\omega_1$-like model of PA.  On the other hand for any completion $T\supseteq$ PA  and countable Scott set
 $S$ there are $2^{\aleph_1}$ non-isomorphic $\omega_1$-like models of $T$ with standard system $S$ all of whose real closures have the same value group.}
\end{Remark}

 We conclude  with a second proof of  Theorem \ref{VG}  that applies a result of Harnik  on models of Presburger arithmetic expandable to $\omega_1$-like models of PA.
 
 \begin{Theorem}[Harnik \cite{harnik}] If $\M$ and $\N$ are $\omega_1$-like models of {\rm PA} with the same standard system, then their ordered additive groups are isomorphic.
\end{Theorem} 

Theorem \ref{VG} now follows directly from

\begin{Lemma} If $\M\models {\rm PA}$, then there is an order reversing isomorphism between $(\M,+)/\Z$ and the value group of $R(\M)$.
\end{Lemma}
\begin{proof} For every $x\in R(\M)$ we can find unique $m\in\M$ and $r\in R(\M)$ such that
$x=2^mr$ and $1\le r<2$.   Note that \begin{itemize}\item if $1\le r,s<2$, then $v(2^mr)=v(2^ns)$ if and only if $m\equiv n \mod \Z$\item $v(2^mr \cdot 2^ns)=v(2^{m+n}rs)$;\item
$v(2^mr)>0$ if and only if $m<\Z$.\end{itemize}
\end{proof}

\section{Real Closures of $\omega_1$-like Models} \label{rcf}

 D'Aquino, Knight and Starchenko conclude in \cite{DKS}  that if $\M$ and $\N$ are countable models of PA with the 
same standard system, then their real closures are isomorphic.  In contrast we prove that  this fails badly for $\omega_1$-like models.
 
 \begin{Theorem}\label{main} Let $\M_0$ be a countable recursively saturated model of {\rm PA}.  There are $2^{\aleph_1}$ $ \omega_1$-like recursively saturated elementary  end extensions
 of $\M_0$ such that the real closures of any two are non-isomorphic.
 \end{Theorem} 

\begin{Remark} 
{\rm Note that all of these models have the same standard system, $\SS(\M_0)$.
If $\M$ and $\N$ are elementarily equivalent $\omega_1$-like recursively saturated models of PA with the same standard system, then $\M\equiv_{\infty,\omega_1}\N$
 (see \cite{KS} 10.2.7).  Thus we in fact have $2^{\aleph_1}$ pairwise $L_{\infty, \omega_1}$-equivalent models of PA with non-isomorphic real closures.}
 \end{Remark}
 
 \smallskip We begin by collecting some of the ingredients we will employ in the proof.
  
 \subsection*{Scott completions}

Let $F$ be an ordered field.  An initial segment $I\subseteq F$ is said to be {\em Dedekindean} if
$I+\epsilon \not \subseteq I$  for all $\epsilon>0$ in $F$. An ordered field $F$ is {\em Scott complete} if every Dedekindean initial segment has a supremum in $F$.  

\begin{Theorem}[Scott \cite{scott}]\label{dana} If $F$ is an ordered field there is a Scott complete ordered field $\hat F$ in which $F$ is dense that is unique up to isomorphism over $F$. Furthermore, if $F$ is real closed, then so is $\hat F$.
\end{Theorem}

We denote the {\em Scott completion} of an ordered field $F$ by $\hat F$.  The Scott completion of $F$ is essentially the set of 
Dedekindean initial segments of $F$.  If $\M\models$ PA, we write $SC(\M)$ for the Scott completion of $Q(\M)$. In general, the Scott completion of an ordered field need not be real closed, but $SC(\M)$ is.

\begin{Lemma}[Schmerl~\cite{schmerl}, Proposition~2.2]\label{scottcompl} If $\M\models$ {\rm PA}  then   $SC(\M)$ is real closed.
\end{Lemma}

In particular, $SC(\M)=\widehat {R(\M)}$. We also observe that any nontrivial definable cut in $Q(\M)$ is Dedekindean.  
Indeed, suppose that $A\subset M^2$ is definable and $I=\{\frac{a}{b}:(a,b)\in A\}$ is a nontrivial  
initial segment of $Q(\M)$. Given $d>0$ we can find a element $c$ such that $\frac{c}{d}\in I$ and $\frac{c+1}{d}\not \in I$, whence $I$ is Dedekindean. 

Recall that $A\subseteq \M$ is a {\em class\/} of $\M$ if for all $b\in \M$ we have that $\{a\in A: a<b\}$ is definable. Let $A$ be a class of $\M$ and 
$$I_A=\{x\in Q(\M): x\le\sum_{a\in A, a\le b} \frac{1}{2^a}\hbox { for some } b\in M\}.$$ 
Then $I_A$ is a Dedekindean initial segment of $Q(\M)$.  
Moreover, if $I\subseteq Q(\M)\cap [0,1)$ is a Dedekindean initial segment,  then there is a class $A$ such that $I=I_A$.

 \subsection*{Facts about PA}
  
Let $\M$ and $\N$ be models of PA. We say that $\M\ee_e\N$ is {\em conservative}  if $X\cap \M$ is definable in $\M$ whenever $X\subset \N$ is definable in $\N$.   
 
If $\M\ee_e\N$ and $A$ is a definable subset of $\N$, then $A\cap \M$ is a class in $\M$, 
and if $A\cap \M$ is not definable in $\M$, then $A\cap \M$ must be unbounded.  These observations follow because the set 
$\{x\in A:x<a\}$ is coded by an element less than $2^a$ and hence coded in $\M$. 

We say that $\M\models$ PA  is {\em rather classless} if every class is definable. The following proposition summarizes three facts about models of PA that we need.

\begin{Theorem}\label{facts}  {\rm (i)} Every model of {\rm PA} has a conservative elementary end extension.

{\rm (ii)} Every countable model of {\rm PA} has a nonconservative elementary end extension.

{\rm (iii)} Suppose that  
$$\M_0\ee_e\M_1\ee_e\cdots\ee_e\M_\alpha\ee_e\cdots,\mbox{\rm\  for $\alpha<\omega_1$}$$ 
is such that each
$\M_\alpha$ is countable and $\M_\alpha=\bigcup_{\beta<\alpha}\M_\beta$ for $\alpha$ a limit ordinal.
If 
$$\{\alpha<\omega_1: \M_{\alpha+1}\mbox{\rm\ is a conservative extension of\ }\M_\alpha\}$$ 
is stationary then $\bigcup_{\alpha<\omega_1}\M_\alpha$ is rather classless.
\end{Theorem}
 
These are, respectively, Theorems~2.2.8,~2.1.7, and~2.2.14 in \cite{KS}.  We shall find it useful to have these results available in a slightly more general setting.  Let $\L^*$ be the language obtained by adjoining a new unary predicate symbol to the language of arithmetic, and let PA$^*$ be the extension of PA in which induction axioms for all $\L^*$ formulas are added.  Each of the assertions in Theorem \ref{facts} holds as well for PA$^*$, where throughout ``definable" is taken to mean ``definable in $\L^*"$; in fact, the context in which these statement appear in \cite{KS} includes this setting.

\medskip We also need the next fact about realizing types in real closures of end extensions.

\begin{Lemma}\label{realize} Suppose that $\M\ee_e\M^\prime\ee_e \M^{\prime\prime}$ are models of\/ {\rm PA}.  Then every 
non-principal type $q$ over $R(\M)$ that is realized in $R(\M^{\prime\prime})$ is already realized in $Q(\M^\prime)$. 
\end{Lemma}

\begin{proof}  If the type $q$ is the type at $\pm\infty$ or of the form
$\{a<v<b: a\in R(\M), a<b\}$ or $\{b<v<a: a\in R(\M), b<a\}$ for some $b\in \M$, then it is easy to see that $q$ is realized 
in $Q(\M^\prime)$. Thus we may assume that $q$ is a cut with no least upper bound or greatest lower bound in $R(\M)$. 

Suppose that $a\in R(\M^{\prime\prime})$ realizes $q$ and 
let $N>\M$. We claim
that $a+\frac{1}{N}$ realizes $q$ as well.  
To this end, let $b\in R(\M)$ be such that $a<b$ and choose $c\in R(\M)$ such that $a<c<b$.
Since $1/N<b-c$, it follows that $a+\frac{1}{N}<c+\frac{1}{N}<b$.  As this is true for every $b\in R(\M)$ 
with $a<b$ we conclude that $a+\frac{1}{N}$ realizes $q$.

Since $\M^{\prime\prime}$ is an integer part of $R(\M^{\prime\prime})$, we have that $Q(\M^{\prime\prime})$ is dense in 
$R(\M^{\prime\prime})$ and thus $q$ is realized in $Q(\M^{\prime\prime})$. So we may assume that $a\in Q(\M^{\prime\prime}).$ 
 
Now let $m\in M$ be such that $a<m$ 
and $d\in \M^\prime\setminus\M$.  In $\M^{\prime\prime}$ we can find $c<dm$ such 
that $\frac{c}{d}<a\le \frac{c+1}{d}$. 
As $\M^\prime\ee_e\M^{\prime\prime}$, we have that $c\in \M^\prime$.
Arguing as above, we see that $\frac{c}{d}$ (as well as $\frac{c+1}{d}$)  realizes $q$, hence $q$ is realized in 
$Q(\M^\prime)$.
\end{proof}

 \subsection*{The basic construction}
 
Fix $\M_0$ a countable model of PA and let $X\subseteq\omega_1$ be stationary.  We build an $\omega_1$-chain of countable models  $$\M_0(X)\ee_e  \M_1(X)\ee_e\cdots\ee_e \M_\alpha(X)\ee_e\cdots\mbox{\ for $\alpha<\omega_1$}$$ such that:
 \begin{itemize}
 \item[(i)] $\M_0(X)=\M_0$;
\item[(ii)]  if $\alpha$ is a limit, then $\M_\alpha(X)=\bigcup_{\beta<\alpha}\M_\beta(X)$;
\item[(iii)] $\M_{\alpha+1}(X)$ is a conservative extension of $\M_\alpha(X)$ if and only if $\alpha\in X$.
\end{itemize}

Note that we can construct such a chain satisfying~(iii) by~\ref{facts}~(i) and~(ii). Let 
$\M(X)=\bigcup_{\alpha<\omega_1}\M_\alpha(X)$.  By Theorem~\ref{facts}~(iii), we have that $\M(X)$ is rather classless. By the remarks following Lemma~\ref{scottcompl},  an initial segment of $Q(\M(X))$ is Dedekindean if and only if it is definable.  In particular $|SC(\M(X))|=\aleph_1$.  We choose a filtration
$\langle S_\alpha(X):\alpha<\omega_1\rangle$ of $SC(\M(X))$ such that each $S_\alpha(X)$ is countable,
 $$S_0(X)\subseteq S_1(X)\subseteq\cdots\subseteq S_\alpha(X)\subseteq \cdots \mbox{\ for $\alpha<\omega_1$,}$$ 
for each limit ordinal $\alpha<\omega_1$ we have $S_\alpha(X)=\bigcup_{\beta<\alpha} S_\beta(X)$, and also 
 $SC(\M(X))=\bigcup_{\alpha<\omega_1}S_\alpha(X)$. 
 
 \begin{Lemma}  Let $X$ and $Y$ be stationary subsets of $\omega_1$ that disagree on closed unbounded sets, that is, $X  \triangle Y$ is stationary. Then $R(\M(X))$ and $R(\M(Y))$ are non-isomorphic.
 \end{Lemma}  
 \begin{proof}  
For a contradiction suppose that  $\sigma:R(\M(X))\rightarrow R(\M(Y))$ is an isomorphism. Note by Theorem~\ref{dana} that 
$\sigma$ extends to an isomorphism of the respective Scott completions, which we also denote as $\sigma$. Each of the sets below is closed and unbounded: 
\begin{itemize}
\item[(a)] $\{\alpha:  \sigma\restriction R(\M_\alpha(X))\mbox{ is an isomorphism from } R(\M_\alpha(X)) \mbox{ to } 
R(\M_\alpha(Y))\}$;
\item[(b)]  $\{\alpha: S_\alpha(X)$ is real closed and $\sigma\restriction S_\alpha(X)$ is an isomorphism onto $S_\alpha(Y)\}$;
\item[(c)]  $\{\alpha: S_\alpha(Z)$ is the set of all Dedekindean initial segments of $Q(\M(Z))$ definable over $\M_\alpha(Z)\}$ where $Z=X$ or $Y$;
\item[(d)]  $\{\alpha: S_\alpha(Z)\cap R(\M(Z))=R(\M_\alpha(Z))\}$, where $Z=X$ or $Y$.
\end{itemize}
Without loss of generality, assume that $Y\setminus X$ is stationary.  We can find an ordinal $\alpha\in Y\setminus X$ such
that $\alpha$ lies in the intersection of the six closed unbounded sets above.

\smallskip  Since $\M_{\alpha+1}(X)$ is a non-conservative extension of $\M_\alpha(X)$, there is a definable
 $A\subset \M_{\alpha+1}(X)$  such that $A\cap \M_{\alpha}(X)$ is not definable in $\M_\alpha(X)$.  Fix $b\in \M_{\alpha+1}(X)\setminus \M_\alpha(X)$. Then 
 $x=\sum_{a\in A, a\le b}\frac{1}{2^a}\in Q(\M_{\alpha+1}(X))$ realizes a Dedekindean cut over $Q(\M_{\alpha}(X))$ that is not definable in $\M_{\alpha}(X)$, and hence  by~(c) and~(d) above no element of  $S_\alpha(X)$ realizes the type of $x$ over 
 $R(M_\alpha(X))$.

The type of $\sigma(x)$ over $R(M_\alpha(Y))$ is Dedekindean and
by Lemma~\ref{realize} is realized in $Q(\M_{\alpha+1}(Y))$.
Since $\M_{\alpha+1}(Y)$ is conservative over $\M_{\alpha}(Y)$, the cut determined by $\sigma(x)$ is definable in 
$\M_{\alpha}(Y)$. Let this cut be given by 
$I=\{\frac{b}{c}: b, c \in \M_{\alpha}(Y) \mbox{\ and\ } \M_{\alpha}(Y)\models \phi(b,c,\bar a)\}$. 
As $I$ is definable, $I\in SC(\M(Y))$, and so, as $I$ is definable
over $\M_{\alpha}(Y)$, by~(c), we have that $I\in S_\alpha(Y)$.
 
We thus have shown that no element of $S_\alpha(X)$ realizes the cut of $x$ over $R(\M_\alpha(X))$,
but that there is an element of $S_\alpha(Y)$ realizing the cut of $\sigma(x)$ over $R(\M_\alpha(Y))$. This contradicts the fact by~(a) and~(b) above  
that $\sigma$ is an isomorphism between $S_\alpha(X))$ and $S_\alpha(Y)$ that restricts to an isomorphism between 
$R(\M_\alpha(X))$ and $R(\M_\alpha(Y))$.
 \end{proof}
 
To obtain the maximum number of  non-isomorphic real closures as in the conclusion of Theorem~\ref{main} we empoly a standard combinatorial lemma. For its proof, see for example, \cite{marker}~5.3.10.
 
 \begin{Lemma} There exists a family $(X_\alpha: \alpha<\omega_1)$ of pairwise disjoint stationary subsets of 
 $\omega_1$.  In addition, if for all $A\subseteq\omega_1$ we put $X_A=\bigcup_{\alpha\in A}X_\alpha$, then 
 $\{X_A:A\subset \omega_1\}$ is a family of $2^{\aleph_1}$ stationary subsets of $\omega_1$ such that $X_A\triangle X_B$ is stationary for all $A\ne B$.
 \end{Lemma}   
 
 \begin{Corollary} Let $\M$ be a countable model of {\rm PA}.  There is a family $(\M_\alpha:\alpha<2^{\aleph_1})$ of $\omega_1$-like elementary end extensions of $\M$ such that
 $R(\M_\alpha)\not\iso R(\M_\beta)$ for $\alpha\ne \beta$.  In particular, all of these models have the same standard system.
 \end{Corollary}
 
 \subsection*{Recursive saturation}
 
To finish the proof of Theorem~\ref{main} we need to modify the basic construction above to build recursively saturated models. This is accomplished by a standard trick. 
 
Let $\M$ be a model of PA. We say that $\Gamma\subset \M$ is a {\em partial satisfaction class for $\M$\/} if 
$$(\M,\Gamma)\models\forall \bar v\left[\phi(\bar v) \leftrightarrow \< \lc \phi\rc,\bar v\> \in \Gamma\right]$$
for all formulas $\phi(\bar v)$, where $\< \cdot,\cdot \> $ is a standard pairing function and 
$\lc\phi\rc$ is a fixed G\"odel coding of formulas.  We say that $\Gamma\subset \M$ is {\em inductive} if $(\M,\Gamma)$ satisfies the induction axiom for formulas in the language in which we adjoin a unary predicate for $\Gamma$, i.e., $(\M,\Gamma)\models$ PA$^*$.

The following lemma is Proposition~1.9.4 of~\cite{KS}.

\begin{Lemma}\label{satisfclass} {\rm (i)} If $\Gamma\subseteq \M$ is an inductive partial satisfaction class, then $\M$ is recursively saturated.

{\rm (ii)} If $\M$ is resplendent, then $\M$ has an inductive partial satisfaction class.
\end{Lemma}

We now apply the foregoing to complete the proof of Theorem~\ref{main}. Let $\M\models$ PA be countable and recursively saturated, and let $\Gamma\subset \M$ be an inductive satisfaction class as guaranteed by Lemma~\ref{satisfclass}(ii).
We now carry out the basic construction starting with $(\M,\Gamma)\models$ PA$^*$.  The corresponding reducts to the language of PA are elementarily equivalent recursively saturated $\omega_1$-like models of PA with standard system 
$\SS(\M)$ and non-isomorphic real closures. With this, Theorem~\ref{main} is proved.

 \end{document}